\documentclass[12pt,epsfig,amsfonts]{amsart}

\usepackage{amsmath,amssymb, amsthm, epsfig}

\newtheorem*{theorem}{Theorem}
\newtheorem{lemma}{Lemma}[section]

\newtheorem{sublemma}[lemma]{Sublemma}
\newtheorem{remark}[lemma]{Remark}

\newtheorem{definition}[lemma]{Definition}

\begin{document}

\thispagestyle{empty}

\author{Hiroki Takahasi and Qiudong Wang}
\address{Institute of Industrial Science, The University of Tokyo, Tokyo
153-8505, JAPAN}
 \email{h{\_}takahasi@sat.t.u-tokyo.ac.jp}
\address{Department of Mathematics, University of Arizona,
Tucson, AZ 85721} \email{dwang@math.arizona.edu}

\thanks{The first named author is partially supported by Grant-in-Aid for Young Scientists (B) of the Japan Society for the Promotion of Science (JSPS), Grant No.23740121, and by
Aihara Project, the FIRST Program from JSPS, initiated by the Council for Science and Technology Policy. 
The second named author is partially
supported by a grant from the NSF}

\title{Nonuniformly Expanding 1D Maps With Logarithmic
Singularities}

\begin{abstract}
For a certain parametrized family of maps on the circle with
critical points and logarithmic singularities where derivatives
blow up to infinity, we construct a positive measure set of
parameters corresponding to maps which exhibit nonuniformly
expanding behavior. This implies the existence of ``chaotic" dynamics
in dissipative homoclinic tangles in periodically perturbed differential
equations.
\end{abstract}
\maketitle

\section{Introduction}

Let $f_{a}: {\mathbb R} \to {\mathbb R}$ be such that
\begin{equation}\label{f1-s1}
f_{a}\colon x\mapsto x+a+L\cdot\ln |\Phi(x)|, \ \  L>0,
\end{equation}
where $a \in [0,1]$ and
$\Phi\colon\mathbb R\to\mathbb R$ is $C^2$ satisfying: (i) $\Phi(x+1) = \Phi(x)$; 
(ii) $\Phi'(x)\neq0$ if $\Phi(x)=0$, (iii) $\Phi''(x)\neq0$ if $\Phi'(x)=0$.
The family $(f_{a})$ induces a
parametrized family of maps from $S^1={\mathbb R} /{\mathbb Z}$
to itself. In this paper we study the abundance of nonuniform hyperbolicity 
in this family of circle maps.

Our study of $(f_{a})$ is motivated by the
recent studies of \cite{W,WOk,WOk11,WO} on homoclinic tangles and
strange attractors in periodically perturbed differential
equations. When a homoclinic solution of a dissipative saddle is
periodically perturbed, the perturbation either pulls the stable
and the unstable manifold of the saddle fix point completely
apart, or it creates chaos through homoclinic intersections. In
both cases, the separatrix map induced by the solutions of the
perturbed equation in the extended phase space is a family of
two-dimensional maps. Taking a singular limit, one obtains a
family of one-dimensional in the form of (\ref{f1-s1}) (with the
absolute value sign around $\Phi(x)$ removed). Let $\mu$ be a
small parameter representing the magnitude of the perturbation and
$\omega$ be the forcing frequency. We have $a \sim \omega \ln
\mu^{-1}$ ({\rm mod} $1$), \ $L \sim \omega$; and $\Phi$ is the
classical Melnikov function (See \cite{WOk,WOk11,WO}).

When we start with {\it two} unperturbed homoclinic loops and
assume symmetry, then the separatrix maps are a family of annulus
maps, the singular limit of which is precisely $f_{a}$ in
(\ref{f1-s1}) (See \cite{W}). If the stable and unstable manifolds
of the perturbed saddle are pulled completely apart by the forcing
function, then $\Phi(x) \neq 0$ for all $x$. In this
case we obtain strange attractors, to which the theory of rank one
maps developed in \cite{WY3} apply. If the stable and unstable
manifold intersect, then $\Phi(x) = 0$ is allowed and the
strange attractors are associated to homoclinic intersections. For
the modern theory of chaos and dynamical systems, this is a case
of historical and practical importance; see \cite{GH,SSTC1,SSTC2}.
To this case, unfortunately, the theory of rank one maps in
\cite{WY3} does not apply because of the existence of the
singularities of $f_{a}$. Our ultimate goal is to develop a
theory that can be applied to the separatrix maps allowing
$\Phi(x) = 0$. This paper is the first step, in which we
develop a 1D theory.

For $f = f_{a}$, let $C(f) = \{ f'(x) = 0 \}$ be the set
of critical points and $S(f) = \{ \Phi(x) = 0 \}$ be the set
of singular points.
In this paper we
are interested in the case $L\gg1$.
As $L$ gets larger,
the contracting region gets smaller and the dynamics is more and more expanding in most of the phase space. Nevertheless, the recurrence of the critical points
is inevitable, and thus infinitesimal changes of dynamics occur when $a$ is
varied.
In addition, the logarithmic nature of the singular set $S$ turns out to 
present a new phenomenon \cite{T} which is unknown to occur for smooth one-dimensional maps
with critical points.

Our main result states that nonuniform expansion prevails for ``most" parameters, provided $L\gg1$.
Let $\lambda=10^{-3}$ and
let $|\cdot|$ denote the one-dimensional Lebesgue measure.
\begin{theorem}
For all large $L$ there exists a set $\Delta=\Delta(L) \subset [0,1)$ of
$a$-values with $|\Delta|>0$ such that if $a \in
\Delta$ then for $f=f_a$ and each $c \in C$, $|(f^n)'(fc)|\geq L^{\lambda n}$
holds for every $n\geq 0$. In addition,  
$|\Delta|\to1$ holds as $L \to \infty$.
\end{theorem}

For the maps corresponding to the parameters in $\Delta$, our argument shows a nonuniform expansion, i,e, for Lebesgue a.e. $x\in S^1$,
$\displaystyle{\varlimsup_{n\to\infty}}\frac{1}{n}\ln|(f_a^n)'x|\geq\frac{1}{3}\ln L.$ In addition, combining our argument with an argument 
in [\cite{WY1} Sect.3] one can construct invariant probability 
measures absolutely continuous with respect to Lebesgue measure (acips).
The main difference from the smooth case is to bound distortions, which 
can be handled with Lemma \ref{dist} in this paper.
A careful construction exploiting the largeness
of $L$ shows the uniqueness of acips and some
advanced statistical limit theorems (See \cite{T}).

Since the pioneering work of Jakobson \cite{J}, there have been quite a few 
number of papers over the last thirty years dedicated to proving the abundance 
of chaotic dynamics in increasingly general families of 
smooth one-dimensional maps 
\cite{BC1,BC2,R,TTY,T1,T2,WY1}. 
Families of maps with critical and singular sets 
were studied in \cite{LT,LV,PRV}. One key aspect of the singularities of our maps that 
has no analogue in \cite{LT,LV,PRV} is that, returns 
to a neighborhood of the singular set can happen very frequently. 
The previous arguments seem not sufficient to deal with points like this.
To avoid problems arising from the logarithmic singularities, and to 
get the asymptotic estimate on the measure of $\Delta$,
we introduce new arguments:

\begin{itemize}

\item our definition of bound periods (see Sect.\ref{s2.3}) incorporates the recurrence pattern of the critical orbits to both $C$ and $S$. 
Thus, the resultant bound period partition depends on the parameter $a$,
and is not a fixed partition, as is the case in \cite{BC1,BC2};

\item to get the asymptotic estimate $|\Delta|\to1$ as $L\to1$,
we need to adandon starting an inductive construction with small intervals
around Misiurewicz parameters. Instead we start with a large parameter set (denoted by $\Delta_{N}$),
which is a union of a finite number of intervals.
This necessitates additional works on establishing uniform hyperbolicity outside of a neighborhood of $C$, which is rather straightforward around Misiurewicz
parameters.
\end{itemize}

The rest of this paper consists of two sections.
In Sect.2 we perform phase space analyses. Building on them, in Sect.3
we construct the parameter set $\Delta$ by induction.
To estimate the measure of the set of parameters excluded at each step, 
instead of Benedicks $\&$ Carleson \cite{BC1,BC2} we elect to follow the approach of Tsujii \cite{T1,T2}, primarily because partitions depend on $a$,
and the
extension of this approach is more transparent in our dealing with
the issues related to the singularities. 
Unlike \cite{BC1,BC2}, the current
strategy relies on a geometric structure of the set of parameters excluded 
at each step. In addition, there is
no longer the need for a large deviation argument, introduced
originally in \cite{BC2} as an independent step of parameter exclusions.

\section{Phase space analysis}\label{s2}
In this section we carry out a phase space analysis. 
Elementary facts on $f_{a}$ are introduced in Sect.
\ref{s2.1}. In Sect.\ref{distortion}
we prove a statement on distortion.
In Sect.\ref{s2.2} we discuss an initial set-up.
In Sect.\ref{s2.3} we introduce three conditions, which wil be taken as assumptions of induction for the construction of the parameter set $\Delta$,
and develop a binding argument.
In Sect.\ref{s2.4} we study
global dynamical properties of maps satisfying these conditions.

\subsection{Elementary facts}\label{s2.1}
 For $\varepsilon > 0$, we use
$C_{\varepsilon}$ and $S_\varepsilon$ to denote the
$\varepsilon$-neighborhoods of $C$ and $S$ respectively. The
distances from $x \in S^1$ to $C$ and $S$ are denoted
as $d_C(x)$ and $d_S(x)$ respectively. We take $L$ as a
base of ${\rm log}(\cdot)$.

\begin{lemma}\label{derivative}
There exist $K_0>1$ and $\varepsilon_0 >0$ such that the following holds 
for all sufficiently large $L$ and $f=f_{a}$:

\begin{itemize}
\item[(a)] for all $x\in S^1$,
$$ K_0^{-1}L
\frac{d_C(x)}{d_S(x)} \leq|f'x| \leq K_0L
\frac{d_C(x)}{d_S(x)}, \ \ \ \ \ \  |f''x| \leq
\frac{K_0L}{d^2_S(x)};
$$

\item[(b)] for all $\varepsilon
>0$ and $x \not \in C_{\varepsilon}$,
$ |f'x|\geq K_0^{-1} L\varepsilon$; 

\item[(c)] for all $x \in C_{\varepsilon_0}$, $K^{-1}_0 L < |f''
x| < K_0 L$.
\end{itemize}
\end{lemma}
\begin{proof} This lemma follows immediately from
$$
f'=1+L \cdot \frac{\Phi'}{\Phi};
 \ \ \ \ \
f''=L \cdot \frac{\Phi''\Phi-
(\Phi')^2}{\Phi^2}
$$
and our assumptions on $\Phi$ in the beginning of the introduction.
\end{proof}

\subsection{Bounded distortion}\label{distortion}
Let $c\in C$, $c_0=fc$, and $n\geq1$. Let
\begin{equation}\label{Theta}
D_n(c_0)=\frac{1}{\sqrt{L}}\cdot\left[\sum_{i=0}^{n-1}
d_i^{-1}(c_0)\right]^{-1} \ \ \ \ \text{where} \ \ \ \ \
d_i(c_0)=\frac{d_C(c_i)\cdot
d_S(c_i)}{J^i(c_0)}.
\end{equation}
\begin{lemma}\label{dist}
For all $x,y \in[c_0-D_n(c_0),c_0+D_n(c_0)]$ we
have
$J^n(x)\leq 2J^n(y),$
provided that $c_i\notin C\cup S$ for every $0\leq i<n$.
\end{lemma}

\begin{proof} Write $D_n$ for $D_n(c_0)$, and let $I
=[c_0-D_n,c_0+D_n]$. Then
$$
\log \frac{J^n(x)}{J^n(y)}=\sum_{j=0}^{n-1}\log \frac{J(f^jx)}{J(f^jy)} \leq\sum_{j=0}^{
n-1}|f^jI| \sup_{\phi\in f^jI}\frac{|f''\phi|}{|f' \phi|}.
$$
Lemma \ref{dist} would hold if for all $j \leq n-1$ we have
$f^jI\cap (S\cup C)=\emptyset$ and
\begin{equation}\label{disteq}
|f^jI|\sup_{\phi\in f^jI}\frac{|f''\phi|}{|f' \phi|}\leq
\log 2 \cdot d_j^{-1}(c_0) \left[\sum_{i=0}^{
n-1}d_i^{-1}(c_0)\right]^{-1}.
\end{equation}
We prove (\ref{disteq}) by induction on $j$. Assume (\ref{disteq})
holds for all $j < k$. Summing (\ref{disteq}) over $j=0,1,\cdots,
k-1$ implies
$
\frac{1}{2} \leq \frac{J^k(\eta)}{J^k(c_0)}\leq 2$
for all $\eta \in I$. We have
\begin{equation}\label{derivative6}
|f^{k}I|\leq 2 J^{k}(c_0) D_n = 2d_{k}^{-1}
\cdot d_C(c_k) d_S(c_k) D_n \leq
\frac{2}{\sqrt{L}}d_C(c_k)d_S(c_k).
\end{equation}
We have $f^k I \cap (C\cup S)
= \emptyset$ from (\ref{derivative6}), and for $\phi\in f^k
I$,
\begin{align*}
|f^kI|\frac{|f''\phi|}{|f'\phi|} &\leq 2
d_k^{-1}  d_C(c_k) d_S(c_k) D_n \cdot \frac{K_0^2}{d_C(\phi) d_S(\phi)} \\
& = 2K_0^2 d_k^{-1}D_n  \cdot \frac{d_C(c_k)
d_S(c_k) }{d_C(\phi) d_S(\phi)}
\leq\frac{2K_0^2}{\sqrt{L}}\cdot d_{k}^{-1} \left[\sum_{i=0}^{n-1}
d_i^{-1}\right]^{-1},
\end{align*}
where we used Lemma \ref{derivative}(a)
for $\frac{|f''\phi|}{|f'\phi|}$ for the first inequality. For
the last inequality we observe that the second factor of the left-hand-side
is $< 2$
by (\ref{derivative6}).  \end{proof}

\subsection{Initial setup}\label{s2.2}
In one-dimensional dynamics, a general strategy for 
constructing positive measure sets of 
``good" parameters is to start an inductive construction
in small parameter intervals, in which orbits of critical points 
are kept out of bad sets for certain number of iterates.
One way to find these intervals is to first look for
Misiuriewicz parameters, for which all critical orbits stay out of
the bad sets under any positive iterate. We would then
confine ourselves in small parameter intervals containing 
the Misiuriewicz parameters, and would eventually prove that the 
Misiuriewicz parameters are Lebesgue density points of the good parameter 
sets. This approach for initial set-ups, however, is with
some drawbacks. First, for a one-parameter family of 
maps with multiple critical points, the Misiuriewicz parameters are relatively 
hard to find because of the need of controlling multiple critical orbits 
with one parameter. Though the argument in \cite{WY2} is readily extended to
cover our family, we are nevertheless up to a
hard start. Second, with the rest of the study confined in a small
parameter interval containing a Misiuriewicz parameter, 
it is not clear how we could
prove the global asymptotic measure estimate ($|\Delta| \to 1$ as
$L \to \infty$) of the theorem.

An alternative route that is made possible by the approach of this
paper is to start with a rather straight forward and
relatively weak assumption. 
Let $\sigma = L^{-\frac16}$.  Let $N$ 
be a large integer independent of $L$. 
For $0\leq n\leq N$, let
$$
\Delta_n =\{a\in [0,1)\colon f_a^{i+1}(C)\cap(C_\sigma\cup
S_\sigma)=\emptyset\text{ for every }0\leq i\leq n\}.
$$
Observe that $\Delta_n$ is a union of intervals.
We start with the following statement, the proof of which is given in Appendix.
\begin{lemma}\label{initial-1}
For any large integer $N$ 
there exists $L_0 = L_0(N)\gg1$ 
such that if $L\geq L_0$, then $
|\Delta_{N}|\geq 1  - L^{-\frac19}.
$
\end{lemma}
Lemma \ref{initial-1} is sufficient for us to move forward. This
approach for initial setups is easier, and leads to the
desired asymptotic measure estimate on $\Delta$ as $L\to\infty$.

\smallskip

We move to the expanding property of the maps corresponding to parameters in
$\Delta_{N}$.  
We frequently use the following notation: for $c\in C$ and $n\geq1$,
$c_0=fc$ and $c_n=f^nc_0$: for $x\in S^1$ and $n\geq1$,
$J(x)=|f'x|$ and $J^n(x)=J(x)J(fx)\cdots J(f^{n-1}x)$.

Let $\alpha=10^{-6}$ and $\delta = L^{-\alpha N}$.
In what follows, we suppose $N$ to be a large integer for which 
$\delta\ll\sigma$, and
the conclusion of Lemma \ref{initial-1} holds. The value of $N$ will be replaced if necessary, but only a finite number of times.
The letter $K$ will be used to denote
generic constants which are independent of $N$ and $L$.

The next lemma, the proof of which is given in Appendix,
ensures an exponential growth of derivatives for 
orbit segments lying outside of $C_{\delta}$. 
\begin{lemma}\label{outside}
There exists $L_1\geq L_0$
such that if $L\geq L_1$ and $f = f_{a}$ is such that $a \in \Delta_{N}(L)$,
then the following holds:
\begin{itemize}
\item[(a)] if $n\geq1$
and $x$, $fx,\cdots, f^{n-1}x\notin C_{\delta}$,
then $J^{n}(x)\geq \delta L^{2\lambda n}$; 
\item[(b)] if
moreover $f^nx\in C_{\delta}$,  then $J^{n}(x) \geq
L^{2 \lambda n}$.
\end{itemize}
\end{lemma}

\noindent {\it Standing assumption for the rest of this section:  \ $L\geq L_1$ and $a \in \Delta_{N}$.}

\subsection{Recovering expansion}\label{s2.3}
For $f = f_a$, $c \in C$ and $n>N$ we introduce three conditions:

\begin{itemize}

\item[(G1$)_{n,c}$] $J^{j-i}(c_i)\geq L \min\{\sigma, L^{-\alpha
i}\}\ \ 0\leq \forall i<\forall j\leq n+1;$

\item[(G2$)_{n,c}$]  $J^i(c_0)\geq L^{\lambda i}\ \ 0<
\forall i\leq n+1;$

\item[(G3$)_{n,c}$] $d_S(c_i)\geq L^{-4\alpha i}\ \ N
\leq \forall i\leq n$.

\end{itemize}

We say $f$ satisfies (G1$)_n$ if (G1$)_{n,c}$ holds for every $c\in C$.
The definitions of (G2)$_n$, (G3)$_n$ are analogous. These conditions are taken as inductive assumptions in the construction of the parameter set $\Delta$.

We establish a recovery estimate of expansion.
Let
$c\in C$, $c_0=fc$ and assume that (G1)$_{n,c}$-(G3)$_{n,c}$. For $p\in[2,n]$, let
$$I_p(c)=\begin{cases} &f^{-1}[c_0+D_{p-1}(c_0),c_0+D_p(c_0))\ \text{ if $c$ is a local minima of $x\to x+a+L\cdot\ln|\Phi(x)|$};\\
& f^{-1}(c_0-D_{p}(c_0),c_0-D_{p-1}(c_0)]\ \text{ if $c$ is a local maxima of $
x\to x+a+L\cdot\ln|\Phi(x)|$}.
\end{cases}$$
By the non-degeneracy of $c$, $I_p(c)$ is the union of two intervals, one at the right of $c$ and the other at the left.
According to Lemma \ref{dist}, if $x\in I_p(c)$ then the derivatives along the orbit of
$fx$ shadow that of the orbit of $c_0$ for $p-1$ iterates. 
We regard the orbit of $x$ as been bound to the
critical orbit of $c$ up to time $p$; and we call $p$ the {\em
bound period} of $x$ to $c$.

\begin{lemma}\label{reclem1}
If
(G1)$_{n,c}$-(G3)$_{n,c}$ holds, then for $p\in[2,n]$ and $x\in
I_{p}(c)$ we have:
\begin{itemize}
\item[(a)] $p\leq\log |c-x|^{-\frac{2}{\lambda}}$;

\item[(b)] if $x\in C_\delta$, then $J^p(x)|\geq
|c-x|^{-1+\frac{16\alpha}{\lambda}}\geq L^{\frac{\lambda}{3}p}$.
\end{itemize}
\end{lemma}

\begin{proof} By definition we have
$$
|c-x|^2\leq D_{p-1}(c_0)\leq L^{-\frac{1}{2}}d_{p-2}(c_0) <
L^{-\frac{1}{2}} J^{p-2}(c_0)^{-1}.$$
Then by (G2),
\begin{equation}\label{low}
|c-x|^2\leq L^{-\frac{1}{2}-\lambda(p-2)}\leq L^{-\lambda p}
,
\end{equation}
from which (a) follows. The second inequality of (b) follows from 
(\ref{low}).

\begin{sublemma}\label{add-lem1-s2.3B}
For $0\leq i \leq n$ we have:
\begin{itemize}
\item[(a)] $d_C(c_i)\geq
K^{-1} \sigma L^{-\alpha i}$; 

\item[(b)] $J^i(c_0)D_{i+1}(c_0)\geq
L^{-1-7\alpha i}$.
\end{itemize}
\end{sublemma}
We finish the proof of Lemma \ref{reclem1} assuming the conclusion of this sublemma.
We have
\begin{align*} J^{p}(x)=
J^{p-1}(fx)J(x)\geq K^{-1}J^{p-1}(c_0)
\cdot L|c-x|\geq K^{-1}J^{p-1}(c_0) \cdot |c-x|^{-1}
D_{p}(c_0),
\end{align*}
where for the first inequality we use Lemma \ref{dist} and Lemma
\ref{derivative}(c), and for the last inequality we use $x\in I_{-p}(c)\cup I_p(c)$. Using Sublemma \ref{add-lem1-s2.3B}(b) for $i = p-1$ we obtain
$$J^p(x)\geq
K^{-1} L^{-1-7\alpha p} |c-x|^{-1}.
$$
Substituting
into this the upper estimate of $p$ in Lemma \ref{reclem1}(a) we obtain
$$
J^p(x)\geq K^{-1}
L^{-1}|c-x|^{-1+\frac{15\alpha}{\lambda}} \geq
|c-x|^{-1+\frac{16\alpha}{\lambda}}.
$$
We have used $|c-x|\leq\delta=L^{-\alpha N_0}$ for the last
inequality. 
\medskip

It is left to prove the sublemma.
(G1) implies $|f'c_i|\geq
L\min\{\sigma,L^{-\alpha i}\}$. Then (a) follows from
Lemma \ref{derivative}(a). As for (b), let $j\in[0,i]$. By definition,
$$
J^i(c_0)d_j(c_0)=
\frac{J^{i}(c_0)}{J^j(c_0)}d_C(c_j)d_S(c_j).
$$
We have: $\frac{J^i(c_0)}{J^j(c_0)}\geq
L^{-\alpha j}\sigma$ from (G1); $d_C(v_j)\geq K^{-1} \sigma
L^{-\alpha j}$ from (a); $d_S(v_j)\geq \sigma
L^{-4\alpha j}$ from (G3). Hence, $J^i(c_0)d_j\geq K^{-1}
\sigma^3L^{-6\alpha j}$, and thus
\begin{align*}
\sum_{j=0}^{i}J^i(c_0)^{-1}d_j^{-1}(c_0) &\leq \sigma^{-3}
L^{7\alpha i}.
\end{align*}
Taking reciprocals implies (b). 
\end{proof}


\subsection{Global dynamical properties}\label{s2.4}
At step $n$ of induction,
we wish to exclude all parameters for which one of (G1-3$)_n$ is violated
for some $c\in C$, and to estimate the measure of the parameters deleted. 
Conditions (G1) (G2), however, can not be used directly as rules for exclusion,
since they do not care about cummulative effects of ``shallow returns".
Hence we introduce a stronger condition, based on 
the notion of \emph{deep returns},  and will use it as a rule for deletion
in Section \ref{s3}.
\medskip

\noindent{\it Hypothesis in Sect.\ref{s2.4}:}
$f =
f_{a}$ is such that $a \in \Delta_{N_0}$. $n\geq N_0$ and
(G1)$_{n-1}$-(G3)$_{n-1}$ hold for all $c\in C$. 
\medskip


For all $c\in C$ we have:

\begin{itemize}

\item[(i)] $f^{i+1}c\notin C\cup S$ for all $0\leq i\leq n$; 

\item[(ii)] for the orbit of $c_0=fc$, the bound period initiated at all
returns to $C_{\delta}$ before $n$ is $\ll n$.

\end{itemize}

\noindent({\it Bound/free structure})
We divide the orbit of
$c_0$ into alternative bound/free segments as follows. Let
$n_1$ be the smallest $j\geq0$ such that $c_j\in C_\delta$.
For $k>1$, we define free return times $n_{k}$ inductively as
follows. Let $p_{k-1}$ be the bound period of $c_{n_{k-1}}$,
and let $n_{k}$ be the smallest $j\geq n_{k-1}+p_{k-1}$ such that
$c_j\in C_\delta$. We decompose the orbit of $c_0$ into
bound segments corresponding to time intervals $(n_k,n_k+p_k)$ and
free segments corresponding to time intervals $[n_k+p_k,n_{k+1}]$.
The times $n_k$ are the {\it free return} times. We have
$$
n_1<n_1+p_1 \leq n_2<n_2+p_2 \leq \cdots.
$$

\begin{definition} 
{\rm Let $c\in C$ and assume that $c_0=fc$
makes a free return to
$C_\delta$ at time $\nu\leq n$.
We say $\nu$ is a {\it deep return} of $c_0$ if for every free return
$i\in [0,\nu-1]$;
\begin{equation}\label{inessential}
\sum_{\stackrel{j\in[i+1,\nu]}{free \ return}}2\log
d_C(c_{j})\leq \log d_C(c_{i}).
\end{equation}}
\end{definition}

We say (R$)_{n,c}$ holds if 
\[\prod_{i\in[0, j]: \ \text{deep 
return}} d_C(c_i)  \geq L^{-\frac{1}{20} \lambda \alpha j}\ \ \text{for every } N_0\leq j\leq n.\]
We say (R$)_n$ holds if (R$)_{n,c}$ holds for every $c\in C$.

\begin{lemma}\label{derive0} If (R$)_{n,c}$ 
holds, then (G1$)_{n,c}$, (G2$)_{n,c}$ hold.
\end{lemma}
\begin{proof}
 The first step is to show
\begin{equation}\label{prop-s3.1}
\prod_{\stackrel{N_0 < i \leq n}{\text{\rm free return }}} \ d_C(c_i)\geq L^{- \alpha n}.\end{equation}
We call a free return {\it shallow} if it is not a deep free
return. Let $\mu\in(0,n)$ be a shallow free return time, and
$i(\mu)$ be the largest deep free return time $< \mu$. We claim
that
\begin{equation}\label{inessential}
\sum_{\stackrel{i(\mu) +1\leq j\leq\mu}{\text{free return}}} \log
d_C(c_j) \geq \log d_C(c_{i(\mu)}).
\end{equation}
We finish the proof of (\ref{prop-s3.1}) assuming (\ref{inessential}). Let $\mu_1$
be the largest free shallow return time in $(0, n]$, and $i_1$ be
the largest deep free return time $< \mu_1$. We then let $\mu_2$
be the largest shallow free return time $< i_1$, and $i_2$ be the
largest deep free return time $<\mu_2$, and so on. We obtain a
sequence of deep free return times $i_1 > i_2 \cdots
> i_q$, and we have
\begin{equation}\label{f4-s3.1}
\sum_{\stackrel{0\leq j\leq n}{\text{shallow return}}} \log
d_C(c_{j}) \geq \sum_{j=1}^q\log d_C (c_{i_j})\geq
\sum_{\stackrel{0\leq i\leq k}{\text{deep return}}}\log
d_C(c_{i})
\end{equation}
where the first inequality is from (\ref{inessential}). We then
have
\begin{equation*}
\sum_{0\leq j\leq n\colon \text{free return}} \log d_C(c_{j}) \geq
2\sum_{\stackrel{0\leq j\leq n}{\text{free return}}} \log
d_C(c_{j})
  \geq \alpha n,
\end{equation*}
where the last inequality is from (R)$_{n,c}$.

\smallskip

To prove (\ref{inessential}), we let $\beta_1$ be the smallest
free return time  $\leq \mu-1$ such that
\begin{equation}\label{deep-add2}
\sum_{\stackrel{\beta_1 +1\leq j\leq\mu}{\text{free return}}}2\log
d_C(c_j) > \log d_C(c_{\beta_1}).
\end{equation}
We claim that no deep free return occurs during the
period $[\beta_1+1,\mu]$. This is because if $i' \in [\beta_1 + 1,
\mu]$ is a deep return, then we must have
$$
\sum_{\stackrel{\beta_1 +1\leq j\leq\mu}{\text{free return}}}2\log
d_C(c_j) \leq \sum_{\stackrel{\beta_1 +1\leq j\leq
i'}{\text{free return}}}2\log d_C(c_j) \leq \log
d_C(c_{\beta_1}),
$$
contradicting (\ref{deep-add2}). If $\beta_1$ is a deep return, we
are done. Otherwise we  find a $\beta_2 < \beta_1$ so that
\begin{equation}\label{deep-add1}
\sum_{\stackrel{\beta_2+1 \leq j\leq \beta_1}{\text{free
return}}}2\log d_C(c_j) > \log d_C(c_{\beta_2}),
\end{equation}
and so on. This process will end at a deep free return time, which
we denote as $\beta_q : = i(\mu)$. (\ref{inessential}) follows
from adding (\ref{deep-add2}), (\ref{deep-add1}) and so on up to
the time for $\beta_q = i(\mu)$. 
\medskip

We first prove (G1)$_{n,c}$. Let $0\leq i < j \leq
n+1$.
Observe that 
the bound periods for all returns to $C_{\delta}$ for
the orbit of $c_0$ up to time $n$ is $\leq \frac{2
\alpha}{\lambda}n \ll n$. This follows from (\ref{prop-s3.1}) and Lemma
\ref{reclem1}(a).
Hence, it is possible to
introduce the bound-free structure starting from $c_i$ to
$c_j$. We consider the following two cases separately.

\smallskip

\noindent {\it Case I: $j$ is free.} \ For free segments we use
Lemma \ref{outside}, and for bound segments we use Lemma
\ref{reclem1}(c). We obtain exponential growth of derivatives from
time $i$ to $j$, which is much better than what is asserted by
(G1$)_{n,c}$.

\smallskip

\noindent{\it Case II: $j$ is bound.} Let $\hat j$ denote the free
return with a bound period $p$ such that $j\in[\hat j+1,\hat
j+p]$. We have $\hat j \leq n$, for otherwise $j>n+1$.
Consequently,
$$
J^{j-i}(c_i)= \frac{J^{\hat
j}(c_0)}{J^i(c_0)} \cdot \frac{J^{\hat j+1}(c_0)}{J^{\hat
j}(c_0)} \cdot \frac{J^j(c_0)}{J^{\hat j +1}(c_0)} >
L^{\frac{1}{3} \lambda(\hat j-i)} \cdot K^{-1} Ld_C(v_{\hat j})
\cdot K^{-1} L^{\lambda(j-\hat j-1)}
$$
where for the last inequality, we use Lemma \ref{reclem1}(c)
combined with Lemma \ref{outside} for the first factor. For the
third factor we use bound distortion and (G2)$_{n-1}$ for the binding critical orbit. It then follows that
$
J^{j-i}(c_0) \geq L^{\alpha(j-i)- \alpha \hat
j}\geq L^{-\alpha i}.
$
Hence (G1$)_{n,c}$ holds.

\smallskip

As for (G2$)_{n,c}$, we introduce the bound-free structure starting from 
$c_0$ to $c_{n+1}$. Observe that
the sum of the lengths of all bound periods for the orbit of $c_0$
up to time $n$ is $\leq \frac{2
\alpha}{\lambda}n \ll n$. This follows from (\ref{prop-s3.1}) and Lemma
\ref{reclem1}(a).
Using Lemma \ref{reclem1}(b) for each bound segment and
Lemma \ref{outside} for each free segment in between two consecutive
free returns, we have
$$
J^{n+1}(c_0) \geq \delta L^{2 \lambda n(1-\frac{\alpha}{\lambda})}\geq  L^{ \lambda n}.
$$
 This completes the proof of Lemma \ref{derive0}. \end{proof}

The next expansion estimate at deep return times will be used in a crucial way in the construction of the parameter set $\Delta$.
\begin{lemma}\label{exp}
If $c \in C$ and $\nu\leq n+1$ is a deep return time of  $c_0=fc$, then
$$
J^{\nu}(c_0) \cdot D_{\nu}(c_0) \geq \sqrt{d_C(c_{\nu})}.
$$
\end{lemma}

\begin{proof}
Let $0<n_1<\cdots<n_t<\nu$ denote all free
returns in the first $\nu$ iterates of $c_0$, with
$p_1,\cdots,p_t$ the corresponding bound periods. Let
$$
\Theta_{n_k}=\sum_{i=n_k}^{n_k+p_k-1}d_i^{-1}(c_0) \ \ \ \text{
and } \ \ \ \Theta_{0}=\sum_{i=0}^{\nu-1}d_i^{-1}(c_0)
-\sum_{k=1}^{t} \Theta_{n_k}.
$$

\smallskip

\noindent {\it Step 1 (Estimate for bound segments):} 
Observe that
\begin{align*}
& \frac{\Theta_{n_k}}{J^{n_k+p_k}(c_0)} =\frac{1}{J^{p_k}(c_{n_k}) d_C(c_{n_k})
d_S(c_{n_k})} + \sum_{i = n_k+1}^{n_k+p_k-1}
\frac{1}{J^{n_k+p_k-i}(c_i) d_C(c_i) d_S(c_i)}.
\end{align*}
To estimate the first term we use Lemma \ref{reclem1}(b) to obtain
\begin{equation}\label{f1-deep}
\frac{1}{J^{p_k}(c_{n_k}) d_C(c_{n_k})
d_S(c_{n_k})}\leq \sigma^{-1}
(d_C(c_{n_k}))^{- \frac{16 \alpha}{\lambda}}.
\end{equation}
To estimate the second term we let $\tilde c$ be the critical
point to which $c_{n_k}$ is bound. By using Lemma \ref{dist} and
(\ref{derivative6}) in the proof of Lemma \ref{dist}, which
implies $d_C(c_i)>\frac{1}{2} d_C(\tilde c_{i-n_k-1})$ and
$d_S(c_i)> \frac{1}{2} d_S(\tilde c_{i-n_k-1})$ for $i \in [n_k+1,
n_k+p_k-1]$, we have
\begin{align*}J^{n_k+p_k-i}(c_i)d_C(c_i)d_S(c_i)
\geq K^{-1} J^{n_k+p_k-i}(\tilde c_{i-n_k-1})d_C(\tilde c_{i-n_k-1})d_S(\tilde c_{i-n_k-1}) \geq
\sigma^2L^{-5\alpha(i-n_k-1)}
\end{align*}
where the last inequality is obtained by using (G1), Lemma
\ref{add-lem1-s2.3B}(a) and (G3) for $\tilde c$. Summing this
estimate over all $i$ and combining the result with (\ref{f1-deep}), 
\begin{equation} \label{sublem2}
|(f^{n_k+p_k})'c_0|^{-1} \Theta_{n_k}\leq
 \sigma^{-1} (d_C(c_{n_k}))^{- \frac{16 \alpha}{\lambda}} + \sigma^{-2}L^{6\alpha
p_k} \leq (d_C(c_{n_k}))^{-\frac{18 \alpha}{\lambda}}.
\end{equation}
Here for the last inequality we use $\sigma \gg \delta$ and
$L^{6\alpha p_k}\leq|d_C(c_{n_k})|^{-\frac{12\alpha}{\lambda}}$
from Lemma \ref{reclem1}(a).

\medskip

\noindent {\it Step 2 (Estimate for free segments):} 
By definition,
$$
\frac{\Theta_0}{J^{\nu}(c_0)} = \sum_{i  \in [0, \nu-1] \setminus
(\cup [n_k, n_k + p-1])}\frac{1}{J^{\nu - i}(c_i)d_C(c_i)
d_S(c_i)}.
$$
Here we can not simply use (G3) for $d_S(c_i)$ in proving
(\ref{quatro}). We observe, instead, that either we have $v_i \not
\in S_{\sigma}$, for which $d_S(v_i)
> \sigma$; or $c_i \in S_{\sigma}$ for which we have

$$
J^{\nu - i}(c_i) =J^{\nu-i+1}(c_{i+1})J(c_i)\geq  
K^{-1} L^{\frac{\lambda}{3} (\nu - i + 1)} (d_S(c_i))^{-1}.
$$
It then follows, by using $d_C(c_i) > \delta$, that
\begin{equation}\label{deep}
\frac{\Theta_0}{J^{\nu}(c_0)} \leq   \sum_{i  \in [0, \nu-1]
\setminus (\cup [n_k, n_k + p-1])} K L^{-\frac{\lambda}{3}
(\nu-i)} (\sigma \delta)^{-1} \leq \frac{1}{\sigma \delta}.
\end{equation}
This estimate is unfortunately not good enough for (\ref{quatro}).
To obtained (\ref{quatro}), we need to use
$C_{\delta^{\frac{1}{20}}}$ in the place of $C_{\delta}$ to define
a new bound/free structure for each free segment out of
$C_{\delta}$. For the new free segments, we can now replace
$\delta$ by $\delta^{\frac{1}{20}}$ in (\ref{deep}); for the bound
segments, we use (\ref{sublem2}) with $d_C
> \delta$. We then obtain
\begin{equation}\label{quatro}
\frac{\Theta_0}{J^{\nu}(c_0)}\leq \frac{1}{\sigma
\delta^{\frac{1}{20}}} +  \sum_{i  \in [0, \nu-1] \setminus (\cup
[n_k, n_k + p-1])} K L^{-\frac{1}{3} \lambda (\nu - i)} \delta^{-
\frac{18 \alpha}{\lambda}} < \frac{1}{\delta^{\frac{1}{3}}}.
\end{equation}

\noindent {\it Step 3 (Proof of the Lemma):} \ From the assumption
that $\nu$ is a deep free return, we have
\begin{equation*}|d_C(c_{n_k})|^{-1}\leq
|d_C(c_{\nu})|^{-2}\prod_{j\colon n_j\in(n_k,\nu)}
|d_C(c_{n_j})|^{-2}.
\end{equation*} Substituting this into
(\ref{sublem2}) gives
\begin{equation}\label{plu}
J^{n_k+p_k}(c_0)^{-1}\Theta_{n_k}\leq
|d_C(c_{{\nu}})|^{-\frac{36\alpha}{\lambda}}\prod_{j\colon
n_j\in(n_k,\nu)}|d_C(c_{n_j})|^{-\frac{36\alpha}{\lambda}}.\end{equation}
Meanwhile, splitting the orbit from time $n_k+p_k+1$ to $\nu$ into
bound and free segments and  we have
\begin{equation}\label{plu2}
J^{\nu-n_k-p_k}(c_{n_k+p_k})^{-1}\leq \left(\prod_{j\colon
n_j\in(n_k,\nu)} J^{p_j}(c_{n_j})\right)^{-1}.\end{equation}
Multiplying (\ref{plu}) with (\ref{plu2}) gives
\begin{align*}
J^{\nu}(c_0)^{-1}  \Theta_{n_k} \leq
|d_C(c_{\nu})|^{-\frac{36\alpha}{\lambda}}\prod_{j: \ n_j \in
(n_k, \nu)} \left(J^{p_j}(c_{n_j}) \cdot
|d_C(c_{n_j})|^{\frac{36\alpha}{\lambda}}\right)^{-1}
\end{align*}
\begin{align*}
\leq |d_C(c_{\nu})|^{-\frac{36\alpha}{\lambda}} \prod_{j: \ n_j
\in (n_k, \nu)} (d_C(c_{n_j}))^{\frac{1}{2}}  \ \leq \
\delta^{(t-k)/2}|d_C(c_{\nu})|^{-\frac{36\alpha}{\lambda}}.
\end{align*} where for the second inequality we use Lemma
\ref{reclem1}(b) for $J^{p_j}(c_{n_j})$, and  for the last we
use $d_C(c_{n_j}) < \delta$. Thus
\begin{align*} \sum_{n_k \in [0, \nu-1]} J^{\nu}(c_{0}) ^{-1}\Theta_{n_k}&\leq
|d_C(c_{\nu})|^{-\frac{36\alpha}{\lambda}}
\sum_{k=1}^{t}\delta^{(t-k)/2}\leq
2|d_C(c_{\nu})|^{-\frac{36\alpha}{\lambda}}.
\end{align*}
Combining this with (\ref{quatro}) we obtain
\begin{align*}J^{\nu}(c_{0})^{-1}D_{\nu}^{-1}=
\sqrt {L}\left( \sum_{1 \leq k \leq t} J^{\nu}(c_{0}) ^{-1}
\Theta_{n_k}+ J^{\nu}(c_{0}) ^{-1} \Theta_0\right)\leq
\frac{1}{\sqrt{d_C(c_{\nu})}}.
\end{align*}
This completes the proof of Lemma \ref{exp}. \end{proof}

\section{Measure of the set of excluded parameters}\label{s3}
The rest of the proof of the theorem goes as follows. For $n > N_0$,
let
$$
\Delta_n = \{ a \in \Delta_{N}: \ \ \text{(R1$)_n$ and (G3$)_n$ hold}\},
$$
and set $\Delta=\cap_{n\geq N}\Delta_n$. This is our parameter set in the theorem. In this section we show that $\Delta$ has positive Lebesgue measure,
and $|\Delta|\to1$ as $L\to\infty$.
To this end, define two parameter sets as follows:
$$E_n=\{a\in\Delta_{n-1}\setminus\Delta_n\colon \text{(R$)_{n}$ fails for $f_a$}\};$$
$$E_n'=\{a\in\Delta_{n-1}\setminus\Delta_n\colon \text{(G3$)_{n}$ fails for $f_a$}\}.$$
Obviously, $\Delta_{n-1}\setminus\Delta_n\subset E_n\cup E_n'$.
We show that the measures of these two sets
decrease exponentially fast in $n$.
Building on preliminary results in Sect.\ref{paradist}, in Sect.\ref{R1} 
we estimate the
measure of $E_n$. In Sect.\ref{R2} we estimate the measure of $E_n'$
and complete the proof of the theorem.

\begin{remark}\label{indhyp}
{\rm The following will be used in the argument:
Let $a\in\Delta_{n-1}$. By the definition of $\Delta_{n-1}$,
 (R$)_{n-1}$ holds, and thus by Lemma \ref{derive0}, (G1$)_{n-1}$  (G2$)_{n-1}$ hold.}
\end{remark}

\subsection{Equivalence of derivatives and distortion in parameter space}\label{paradist} 
For $c \in C$ and $i\geq0$, we define $\gamma_i^{(c)}: \Delta_{N} \to S^1$ by
letting $\gamma_i^{(c)}(a) = f^{i+1}_a(c)$.
In this subsection we denote $c_i(a)=\gamma_i^{(c)}(a)$
and
$\tau_i(a) = \frac{d
c_i(a)}{da}$.  
\begin{lemma}\label{trans}
Let $a\in\Delta_{n-1}$.
Then, for
all $c\in C$ and $k \leq n$ we have
$$
\frac{1}{2}\leq\frac{\left|\tau_k(a)\right|}
{\left|(f_a^k)'c_0(a)\right|}\leq 2.
$$
\end{lemma}
\begin{proof}
We have 
\begin{equation}\label{f1-s3.2A}
\tau_{k}(a)=1+f_a'c_{k-1}(a)\cdot\tau_{k-1}(a).
\end{equation}
Using this inductively and then dividing the result by
$(f_a^{k})'c_0(a)$, which is nonzero by (G2$)_{n-1}$
we obtain
$$
\frac{\tau_{k}(a)}{(f_a^{k})'c_0(a)}=1+\sum_{i=1}^{k}\frac{1}
{(f^i)'c_0(a)}.
$$
This lemma then follows from applying (G2)$_{n-1}$. 
\end{proof}

For $a_*\in[0,1)$, $c \in C$ and $c_0=fc$, define
$$
I_n(a_*,c)=[a_*-D_{n}(c_0), a_*+D_{n}(c_0)]
$$
where $
D_n(c_0)$
is the same as the one in (\ref{Theta}) with $f = f_{a_*}$. 

\begin{lemma}\label{samp}
Let $a_*\in\Delta_{n-1}$.
For all $c\in C$, $a \in I_{n}(a_*,c)$
and $k \leq n$ we have
$$
\frac{1}{2} < \frac{|\tau_{k}(a)|}{|\tau_{k}(a_*)|}\leq 2.
$$
\end{lemma}
\begin{proof} Let $k \leq n$. To prove this lemma we
inductively assume that, for all $j < k$,
\begin{equation}\label{samp-induct}
\frac{|\tau_j(a)|}{|\tau_j(a_*)|}\leq 2, \ \ \ \text{for all} \ a
\in I_{n}(a_*, c).
\end{equation}
We then prove the same estimate for $j = k$. Write $I_{n}$ for $I_{n}(a_*,c)$. For all $a \in I_{n}$
we have 
\begin{align*}
\left|\log\frac{|\tau_{j+1}(a)|}{|\tau_{j+1}(a_*)|}-
\log\frac{|\tau_{j}(a)|}{|\tau_{j}(a_*)|}\right| \leq&
\left|\log\frac{|\tau_{j+1}(a)|}{|\tau_j(a)|}-
\log|f_{a_*}'c_j(a_*)|\right|\\
+\left|\log\frac{|\tau_{j+1}(a_*)|}{|\tau_j(a_*)|}-
\log|f_{a_*}'c_j(a_*)|\right|&\leq (I)_a +(I)_{a_*} + (I\!I),
\end{align*}
where
$$
(I)_a = \left|\log\frac{\tau_{j+1}(a)} {\tau_j(a)}-
\log|(f_{a})'c_j(a)|\right|, \ \ \ (I\!I) = \left|\log
\frac{(f_{a})'c_j(a)}{(f_{a_*})'c_j(a_*)} \right|.
$$
We claim that
\begin{equation}\label{twoa}
|(I\!I)|\leq 2L^{-\frac{1}{3}}\cdot d_j^{-1}\left[\sum_{i=0}^{n-1}d_i^{-1}\right]^{-1},
\end{equation}
where $d_i$ is the same as the one in (\ref{Theta}) with $f=f_{a_*}$.

To prove (\ref{twoa}), first we use (\ref{samp-induct}) and Lemma
\ref{trans} to obtain
\begin{align}\label{R1add}
|\gamma_j^{(c)}(I_{n})|\leq 2|\tau_j(a_*)||I_{n}|\leq
4|(f_{a_*}^j)' c_0(a_*)||I_{n}| \leq
2L^{-\frac{1}{2}}d_C(c_j(a_*))d_S(c_j(a_*)).
\end{align}
This implies $d_S(c_j(a))\geq \frac{1}{2} d_S(c_j(a_*))$ for all
$a\in I_{n}$. Thus from Lemma
\ref{derivative}(b),
$$|f''|\leq \frac{KL}{d_S(c_j(a_*))^{2}}\quad\text{on $\gamma_j^{(c)}(I_{n})$.}$$
It then follows that
$$
\left|f_{a}'c_j(a)-f_{a_*}'c_j(a_*)\right| =
\left|f_{a_*}'c_j(a)-f_{a_*}'c_j(a_*)\right| \leq
\frac{KL}{d_S(c_j(a_*))^{2}}|c_j(I_{n})| \leq
\frac{KL}{d_S(c_j(a_*))^{2}}|\tau_j(a_*)| \cdot |I_{n}|,
$$
where (\ref{samp-induct}) is again used for the last inequality.
We have then
\begin{equation*}
\left|f_{a}'c_j(a)-f_{a_*}'c_j(a_*)\right| \leq
\frac{KL^{\frac{1}{2}}d_C(c_j(a_*))}{d_S(c_j(a_*))}d_j^{-1}
\left(\sum_{i=0}^{n-1} d_i^{-1}\right)^{-1}\leq L^{-\frac{1}{3}}|(f_{a_*}')c_j(a_*)|d_j^{-1}\left(\sum_{i=0}^{n-1} d_i^{-1}\right)^{-1},
\end{equation*}
where for the last inequality we used Lemma \ref{derivative}(a).
(\ref{twoa}) follows directly from the last estimate.

As for $(I)_a$, we have from (\ref{f1-s3.2A}),
\begin{equation}\label{onea}
(I)_a\leq \log \left(1+ \frac{1}{|(f_{a})'c_j(a)| \cdot
|\tau_j(a)|}\right) < \frac{1}{|(f_{a})'c_j(a)| \cdot|\tau_j(a)|}
< L^{-\frac{\lambda}{3}j}
\end{equation}
where for the last estimates we use the inductive assumption
(\ref{samp-induct}) and Lemma \ref{trans} for $|\tau_j(a_*)|$. We
also use
$$
|(f_{a})'c_j(a)| \geq\frac{1}{2}|(f_{a_*})'c_j(a_*)| \geq \frac{
L}{2}\cdot\min\{\sigma,L^{-\alpha j} \}
$$
where the second inequality follows from
from (G1). Then $\frac{|\tau_k(a)|}{|\tau_k(a_*)|}\leq 2$ now
follows from combining (\ref{onea}) and (\ref{twoa}) for all $j <
k$. \end{proof}

\subsection{Exclusion on account of (R)}\label{R1}
To estimate the measure of parameters excluded due to (R$)_n$, we divide the critical orbit $\{ v_i, \ i \in [0, n] \}$ into
free/bound segments for $a \in \Delta_{n-1}$, $c \in C$, and let
$t_1 < t_2 < \cdots < t_{q} \leq n$ be the consecutive times for
{\it deep free returns} to $C_{\delta}$. Let $c^{(i)}$ be the
corresponding binding critical point at time $t_i$,  $r_i$ is the unique
integer such that $|c_i-f^{\nu_i}x_*|\in(L^{-r_i},L^{-r_i+1}].$
We call
$$
{\bf i}: = (t_1, r^{(1)}, c^{(1)}; \ t_2, r^{(2)}, c^{(2)}; \
\cdots; \ t_q, r^{(q)}, c^{(q)})
$$
the {\it itinerary} of $v_0 = f_a(c)$ up to time $n$.

Let $E_n(c, {\bf i})$ denote the set of
all $a\in E_n$ for which 
(R$)_{n,c}$ fails at time $n$, and
the itinerary for
$c_0(a)=f_ac$ up to time $n$ is ${\bf i}$.
To estimate the measure of $E_n$, 
we first estimate the measure of $E_n(c, {\bf i})$.
We then combine it with a bound on the nubmer of all feasible
itineraries.

\begin{lemma}\label{pro-R1}
$|E_n(c, {\bf i})| \leq L^{-\frac{1}{3} R}$, where $R = r_1 + r_2
\cdots + r_q$.
\end{lemma}
\begin{proof}
First of all, for $a_* \in R_n(c, {\bf i})$ and $1\leq k\leq q$
 we define a parameter interval
$I_{t_k}(a_*)$ as follows. 
 If $|\gamma_{t_k}^{(c)}(I_n(a_*,
c))|<\frac{1}{4}$, then we let $I_{t_k}(a_*)=I_{t_k}(a_*, c)$.
Otherwise, take $I_{t_k}(a_*)$ to be the interval of length
$\frac{1}{10|\gamma_{t_k}^{(c)} (I_{t_k}(a_*, c))|}|{I}_{t_k}(a_*, c)|$
centered at $a_*$.

The proof of Lemma \ref{pro-R1} is outlined as follows.
For a compact interval $I$ centered at
$a$ and $r>0$, let $ r\cdot I$ denote the interval of length
$r|I|$ centered at $a$.
For each $k\in[1,q]$, we choose a countable subset 
$\{a_{k,i}\}_{i}$ of $E_n(c,\bold i)$ with the following properties:

\begin{itemize}

\item[(i)] the intervals $\{I_{t_k}(a_{k,i})\}_{i}$ are pairwise
disjoint and $E_n(c,\bold i)\subset \bigcup _{i}L^{-r_k/3}\cdot
{I}_{t_k}(a_{k,i});$

\item[(ii)] for each $k\in[2,q]$ and $a_{k,i}$ there exists
$a_{k-1,j}$ such that ${I}_{t_k}(a_{k,i})\subset
2L^{-r_{k-1}/3}\cdot{I}_{{t_{k-1}}}(a_{k-1,j})$.
\end{itemize}
Observe that the desired estimate follows from this.
\smallskip

For the definition of such a subset we need two combinatorial
statements.
 The following elementary fact from
Lemma \ref{exp} is used in the proofs of these two sublemmas:
if $a\in E_n(c,\bold i)$, then $(\gamma_{t_k}^{(c)}|I_{t_k}(a))^{-1}(c^{(k)})$ consists of a single point and is contained in 
$L^{-r_k/3}\cdot I_{t_k}(a)$.

\begin{sublemma}\label{lem1}
If $a$, $a'\in E_n(c,\bold i)$ and $a'\notin I_{t_k}(a)$, then
$I_{t_k}(a)\cap I_{t_k}(a')=\emptyset$ .
\end{sublemma}

\begin{proof}
Suppose $I_{t_k}(a)\cap I_{t_k}(a')\neq\emptyset$. Lemma \ref{dist} gives
$|I_{t_k}(a)|\approx |I_{t_k}(a')|$. This and $a'\in I_{t_k}(a)$ imply
$(\gamma_{t_k}^{(c)}|I_{t_k}(a))^{-1}(c^{(k)})\neq(\gamma_{t_k}^{(c)}|I_{t_k}(a'))^{-1}(c^{(k)}).$ On
the other hand, by the definition of the intervals $I_{t_k}(\cdot)$,
$\gamma_{t_k}^{(c)}$ is injective on $I_{t_k}(a)\cup I_{t_k}(a')$. A
contradiction arises.
\end{proof}

\begin{sublemma}\label{lem2} If $a$, $a'\in E_n
(c,\bold i)$
 and
$a'\in L^{-r_k/3}\cdot I_{k}(a)$, then $I_{t_{k+1}}(a')\subset
2L^{-r_k/3}\cdot I_{t_k}(a)$.
\end{sublemma}

\begin{proof}
We have $(\gamma_{t_k}^{(c)}|I_{t_k}(a))^{-1}(c^{(k)})
\notin I_{t_{k+1}}(a')$, for
otherwise the distortion of $\gamma_{\nu_{k+1}}^{(c)}$ on $I_{t_{k+1}}(a')$ is
unbounded. This and the assumption together imply that one of the
connected components of $I_{t_{k+1}}(a')-\{a'\}$ is contained in
$L^{-r_k/3}\cdot I_{t_k}(a)$. This implies the inclusion.
\end{proof}

We are in position to choose subsets $\{a_{k,i}\}_{i}$ satisfying
(i) (ii). Lemma \ref{lem1} with $k=1$ allows us to pick a subset
$\{a_{1,i}\}$ such that the corresponding intervals $
\{I_{t_1}(a_{1,i})\}$ are pairwise disjoint, and altogether cover
$E_n(c,\bold i)$. Indeed, pick an arbitrary $a_{1,1}$. If $
I_{t_1}(a_{1,1})$ covers $E_n(c,\bold i)$, then the claim holds.
Otherwise, pick $a_{1,2}\in E_n(c,\bold i)-{I}_{t_1}(a_{1,1})$. By
Lemma \ref{lem1}, $I_{t_1}(a_{11}),{I}_{t_1}(a_{12})$ are disjoint.
Repeat this. By Lemma \ref{lem1}, we end up with a countable number of 
pairwise disjoint
intervals. To check the inclusion in (i), let $a\in
{I}_{1}(a_{1i})-L^{-r_1/3}\cdot{I}_{1}(a_{1i})$. By Lemma
\ref{exp}, $|f_a^{t_1+1}c-c^{(1)}|\gg L^{-r_1}$ holds. Hence $x\notin
E_n(c,\bold i)$.

Given $\{a_{k-1,j}\}_j$, we choose $\{a_{k,i}\}_i$ as follows. For
each $a_{k-1,j}$, similarly to the previous paragraph it is
possible to choose parameters $\{a_{m}\}_m$ in $E_n(c,\bold i )\cap
L^{-r_{k-1}/3}\cdot{I}_{t_{k-1}}(a_{k-1,j})$ such that the
corresponding intervals $\{I_{t_k}(a_m)\}_m$ are pairwise disjoint and
altogether cover $E_n(c,\bold i )\cap
L^{-r_{k-1}/3}\cdot{I}_{t_{k-1}}(a_{k-1,j})$. In addition, Lemma
\ref{lem2} gives $\bigcup_m{I}_{t_k}(a_{m})\subset 2L^{-r_{k-1}/3}
 \cdot{I}_{t_{k-1}}(a_{k-1,j}).$ Let $\{a_{k,i}\}_{i}=
\bigcup_{j}\{a_{m}\}$.
 This finishes the proof of Lemma
\ref{pro-R1}. \end{proof}

\begin{lemma}\label{boundlem2}
Assume that $f_a$ is such that $a \in \Delta_{N}$. Then the
lengths of any given bound period due to a return to $C_\delta$ is
$\geq \frac{1}{2} \alpha N$.
\end{lemma}
\begin{proof} From $a \in \Delta_{N}$ and Lemma
\ref{derivative}(a) we have $J^{[\frac{1}{2} \alpha
N]}(c_0)\leq(\sigma^{-1} L)^{\frac{1}{2} \alpha N}$ where
$\sigma = L^{-\frac{1}{6}}$. It then follows that
$$
D_{[\frac{1}{2} \alpha
N]}(c_0)\geq\frac{L^{-\frac{1}{2}}\sigma^2} {J^{[\frac{1}{2}
\alpha N]}(c_0)}\geq L^{-\frac{1}{2}}\sigma^2
(L^{-1}\sigma)^{\frac{1}{2} \alpha N}\gg \delta.
$$
This implies the lemma.  
\end{proof}

Observe that
$E_n= \bigcup E_n(c, {\bf i}),
$
where the union runs over all $c \in C$ and all
feasible itineraries ${\bf i} = (t_1,
r^{(1)}, c^{(1)}; \cdots; t_q, r^{(q)},
c^{(q)})$.
By Lemma \ref{boundlem2} we have
$
q \leq \frac{2 n}{\alpha N}.
$
We also have 
$
R = r_1 + r_2 + \cdots + r_q > \frac{\lambda \alpha n}{20}.
$
 The number of choices for $(r_1,
\cdots r_q)$ satisfying $r_1 + \cdots + r_q = R$ is $
\left(\begin{smallmatrix} R+q\\
q\end{smallmatrix}\right)$, and the possible number of choices for
the deep free return times is $\left(\begin{smallmatrix} n \\
q\end{smallmatrix}\right)$.
By Stirling's formula for
factorials, 
$
\left(\begin{smallmatrix} n \\
q\end{smallmatrix}\right) \leq e^{\beta(N) n}$ and $\left(\begin{smallmatrix} R+q \\
q\end{smallmatrix}\right) \leq e^{\beta(N) R}$,
where $\beta(N) \to 0$ as $N \to \infty$.
Using these and Lemma \ref{pro-R1} we conclude
that
\begin{equation}\label{Delete-R1}
|E_n| \leq  \sum_{1
\leq q \leq \frac{2 n}{\alpha N}}  \sum_{R > \frac{\lambda \alpha n}{20}}(\# C)^q \left(\begin{matrix} n \\
q\end{matrix}\right) \left(\begin{matrix} R+q \\
q\end{matrix}\right) L^{- \frac{R}{3} }\leq L^{-\frac{\lambda \alpha n}{100}},
\end{equation}
where the last inequality holds for sufficiently large
$N$.

\subsection{Exclusion on account of (G3)}\label{R2}
Estimates for exclusions due to
(G3$)_n$ are much simpler.

\begin{lemma}\label{wrap2}
For $a_* \in \Delta_{n-1}$, $c \in C$,
$
|\gamma_{n}^{(c)}(I_{n}(a_*, c))|\geq L^{-3\alpha n}$.
\end{lemma}
\begin{proof} Write $f=f_{a_*}$, $c_i=f^{i+1}c$.
Since $a_* \in \Delta_{n-1}$, (G1)$_{n-1}$
and (G2)$_{n-1}$ hold for $c_0$ by Lemma \ref{derive0}. Using
Lemma \ref{trans} and \ref{samp}, we have
$$
|\gamma_n^{(c)}(I_n(a_*, c))| \geq 4 J^{n}(c_0) \cdot |I_n(a_*,c)| \geq \frac{1}{\sqrt{L}} \left(\sum_{i=0 }^{n-1}
(J^{n}(c_0) d_i)^{-1}\right)^{-1}.
$$
If $c_i\notin S_\sigma$, then
\begin{align*}J^{n}(c_0)d_i&=
J^{n-i}(c_i)d_C(c_i)d_S(c_i)\geq K L^{-2\alpha i}\sigma.
\end{align*}
If $c_i\in S_\sigma$, then $d_S(c_i)\geq |f'c_i|^{-1}$ from Lemma
\ref{derivative}(a), and we have
\begin{align*}J^{n}(c_0)d_i&=
J^{n-i}(c_i)d_C(c_i)d_S(c_i)\geq J^{n-i-1}(c_{i+1})
d_C(c_i)\geq L^{-\alpha(i+1)}\sigma
\end{align*}
where (G1)$_{n-1}$ is used for the last inequality. Hence we
obtain
\begin{align*}
|\gamma_n^{(c)}(I_n(a_*, c))| >
\frac{1}{\sqrt{L}}\left(\sum_{i=0}^{n-1}K L^{2\alpha i} \sigma
\right)^{-1}\geq L^{- 3\alpha n}.
\end{align*}
\end{proof}

We now estimate the measure of $E_n'$. For $c\in C$ and $s\in S$,
let $E_n'(c,s)$ denote the set of all $a\in E_n'$ such that 
$d(\gamma_n^{(c)}(a),s)< L^{-4 \alpha n}$. 
For $a\in E_n'(c,s)$,
define an 
interval $I_n(a)$ centered at $a$, similarly to the definition of
$I_{t_k}(\cdot)$ in the beginning of the proof of Lemma \ref{pro-R1}
(replace $t_k$ by $n$).
 We
claim that:

\begin{itemize}
\item if $a\in E_n'(c,s)$, then 
$I_n(a)\setminus L^{-\alpha n/2}\cdot I_n(a)$ does not intersect $E_n'(c,s)$;

\item if $a,\tilde a\in E_n'(c,s)$ and $\tilde a\notin
I_n(a)$, then $I_n(a)\cap I_n(\tilde a)=\emptyset$.
\end{itemize}

The first item follows from  Lemma \ref{wrap2} and Lemma \ref{samp}.
The second follows from the injectivity argument used in the proof of 
Lemma \ref{lem2}. It follows that
$
|E_n'(c, s)| \leq L^{-\frac{1}{2}\alpha n},
$
and threfore
\begin{equation}\label{delete-D2}
|E_n'| < \# C \# S \cdot L^{-\frac{1}{2} \alpha
n} < L^{- \frac{1}{3} \alpha n}.
\end{equation}
(\ref{Delete-R1}), (\ref{delete-D2}) and Lemma \ref{initial-1} altogether yield
$|\Delta|\to1$ as $L\to\infty$.

\appendix

\section{Proofs of Lemmas \ref{initial-1} and \ref{outside}}
In this appendix we prove Lemmas \ref{initial-1} and
\ref{outside}. Since various ideas and technical tools developed
in the main text are needed, the correct order for the reader is
to go over the main text first before getting into the details of
these two proofs.

\subsection{Proof of Lemma \ref{initial-1}.} 
Let $f=f_{a_*}$ and $c\in C$. Let $I_n(a_*,c)=[a_*-D_n(c_0),a_*+D_n(c_0)]$,
where $D_n(c_0)$ is the same as the one in (\ref{Theta}) with $f=f_{a_*}$.
 Let $\gamma_n^{(c)}(a)=f_a^{n+1}c$
and $\tau_n^{(c)}(a) =
\frac{d}{da} \gamma_n^{(c)}(a)$.

\begin{lemma}\label{lem1-appA}
Let $1\leq n\leq N_0$, $f=f_{a_*}$,
$a_* \in \Delta_{n-1}$ and let $c \in C$. We have:

\begin{itemize}

\item[(a)] $J^{j-i}(c_i) \geq (K^{-1} L
\sigma)^{j-i} $ for all $0 \leq i < j \leq n$;

\item[(b)] $
J^n (x)\leq 2J^n(y)$ for all $x,y
\in[c_0-D_n(c_0),c_0+D_n(c_0)]$;

\item[(c)] 
$\frac{1}{2}\leq\frac{|\tau_n^{(c)}(a_*)|}{J^n(c_0)}\leq 2$;

\item[(d)] 
$\frac{1}{2}\leq\frac{|\tau_n^{(c)}(a)|}{|\tau_n^{(c)}(a_*)|}\leq2$
for all $a\in I_n(a_*,c)$;

\item[(e)] 
$|\gamma_n^{(c)}(I_n(a_*,c))|\geq L^{1/7}\sigma$.
\end{itemize}
\end{lemma}
\begin{proof} Item (a) follows directly from Lemma
\ref{derivative}. (b) is in Lemma
\ref{dist}, (c) is in Lemma \ref{trans} and (d)
is in Lemma \ref{samp}.
As for (e), let $i\in[0,n-1]$. 
(a) gives
\begin{align*}J^n(c_0)d_i(c_0)=
J^{n-i}(c_i)
d_C(c_i)d_S(c_i)\geq
(K_0^{-1}L\sigma)^{n-i}\sigma^2,\end{align*} 
where $d_i(c_0)$ is the same as the one in (\ref{Theta})
with $f=f_{a_*}$.
Taking reciprocals
and then summing the result over all $i\in[0,n-1]$ we have
$$
\sum_{i=0}^{n-1}J^n(c_0)^{-1}d_i(c_0)^{-1}
\leq\frac{K_0}{L\sigma^3}.$$ Recall that $\sigma = L^{-\frac16}$.
We have
$$|\gamma_n^{(c)}(I_n(a_*,c))|^{-1}\leq KJ^n(c_0)^{-1}D_n^{-1}(c_0)= K
\sqrt{L}\cdot
\sum_{i=0}^{n-1}J^n(c_0)^{-1}d_i(c_0)^{-1}\leq
\frac{1}{L^{\frac{1}{7}}\sigma},$$ where the first inequality
follows from (c) and (d). \end{proof}

For $c \in C$, $s \in C \cup S$, let $E_n(c,s)$ denote the set of all
$a\in\Delta_{n-1}\setminus\Delta_n$ so that $d(\gamma_n^{(c)}(a), s)\leq\sigma$. For $a\in E_n(c, s)$, 
define a parameter interval
$I_n(a)$ as follows. If $|\gamma_n^{(c)}(I_n(a,
c))|<\frac{1}{4}$, then we let $I_n(a)=I_n(a, c)$.
Otherwise, define $I_n(a)$ to be the interval of length
$\frac{1}{10|\gamma_n^{(c)} (I_n(a, c))|}|{I}_n(a, c)|$
centered at $a$.
We claim that:

\begin{itemize}
\item if $a\in E_n(c,s)$, then 
$I_n(a)\setminus L^{-1/8}\cdot I_n(a)$ does not intersect $E_n(c,s)$;

\item if $a,\tilde a\in E_n(c,s)$ and $\tilde a\notin
I_n(a)$, then $I_n(a)\cap I_n(\tilde a)=\emptyset$.
\end{itemize}
The first item follows from Lemma
\ref{lem1-appA}(c)-(e). The second follows from the fact that
the map $\gamma_n^{(c)}$ is injective on $I_n(a)$.
Observe that from
this we have $|E_n(c, s)| < L^{-\frac{1}{8}}$, and it follows
that
$
|\Delta_{n-1} \setminus \Delta_n| < \#C (\#C + \#S) L^{-\frac18},
$
and we obtain
$$
|\Delta_{N}| \geq 1 - \sum_{n=1}^{N} |\Delta_{n-1}
\setminus \Delta_n| \geq 1 - L^{-\frac19}.
$$
This completes the proof of 
Lemma \ref{initial-1}. \qed

\subsection{Proof of Lemma \ref{outside}} We need
the materials in Sect.\ref{s2.3} and the definition of
the bound/free structure in Sect.\ref{s2.4}.

Let $\delta_0 = L^{-\frac{11}{12}}$. For the orbit of $x$ lying out
of $C_\delta$ we first introduce the bound/free structure of Sect.\ref{s2.4} by using $C_{\delta_0}$ in the place of $C_{\delta}$.
We then prove a result that is similar to Lemma \ref{reclem1},
from which Lemma \ref{outside} would follow directly. To make this
argument work, we first need
to show that the bound periods on $C_{\delta_0} \setminus
C_{\delta}$, as defined in Sect.\ref{s2.3}, are $\leq 
N$. Let $I_{p}(c)$ and the bound period be defined
the same as in Sect.\ref{s2.3}. 
\begin{lemma}\label{cover}
If $f = f_a$ is such that $a \in \Delta_{N}$, then
for each $c\in C$ we have $[c-\delta_0,c-\delta]\cup [c+\delta,c+\delta_0]\subset \bigcup_{1\leq p\leq N}I_{p}(c)$.
\end{lemma}
\begin{proof} Let $c_0 = fc$. It suffices to show
that
\begin{equation}\label{2.3.1}
L^{-1}D_{N}(c_0)<\delta^2<\delta_0^2 < L^{-1}D_1(c_0)
\end{equation}
where $D_n(c_0)$ is as in (\ref{Theta}).
Since $a\in\Delta_{N}$ we have $J^{N-1}(c_0)\geq
(KL\sigma)^{N-1}$ from Lemma \ref{derivative}. It then follows
that
$$
D_{N}(c_0)\leq d_{N-1}(c_0) \leq
(KL\sigma)^{-N+1}<L\delta^2.
$$
On the other hand we have
$$
D_1(c_0)=\frac{1}{\sqrt{L}}\cdot d_C(c_0)d_S(c_0)\geq
L^{-\frac{1}{2}}\sigma^2,
$$
from which the last inequality of (\ref{2.3.1}) follows directly.
\end{proof}

With the help of Lemma \ref{cover}, we know that the bound/free
structure for the orbit of $x$ out of $C_\delta$ is well-defined.
\begin{lemma}\label{initial}
Let $p\geq2$ be the bound period for $y\in
C_{\delta_0}\setminus C_\delta$. Then:
\begin{itemize}
\item[(a)] \ for $i\in[1,p]$, we have $d_C(f^{i}y)>\delta_0$ and
$|(f^{i-1})'(c_0)| D_{i}(c_0) \geq L^{-\frac12}\sigma^{2}$;

\item[(b)] $J^{p}(y)\geq K^{-1} L^{-\frac12}\sigma^2
|c-y|^{-1}\geq K^{-1} L^{-\frac12} \sigma^2
(K_0^{-1}L\sigma)^{p-2}$;

\item[(c)] $J^{p}(y) \geq L^{\frac{1}{300}p}.$
\end{itemize}
\end{lemma}
\begin{proof} (a) is a version of Sublemma
\ref{add-lem1-s2.3B}. The estimates are better because here we
have $f^{i-j}(c_j) > (KL\sigma)^{i-j}$, and $d_C(c_i), d_S(c_i)>
\sigma$. As for (b) we have
$$
J^{p}(y)\geq K^{-1} L |c-y|J^{p-1}(c_0)| \geq
K^{-1}
  |c -y|^{-1} J^{p-1}(c_0) D_{p}(c_0),
$$
where for the last inequality we use $D_{p}(c_0) \leq K L |c -
y|^2$ by the definition of $p$. The first inequality of (b)
then follows by using (a). The second inequality of (b) follows from 
$$|c-y|\leq D_{p-1}(c_0)\leq\frac{d_{p-2}(c_0)}{\sqrt{L}}
\leq\frac{(K_0^{-1}L\sigma)^{-p+2}}{\sqrt{L}}.$$
Here, last inequality is because $a\in\Delta_{N}$. 

If $p \geq 10$, then (c) is much weaker than
the second inequality of (b). If $p < 10$,
then (c) follows from the first inequality of (b) and the fact that
$
|c - y|^{-1} \geq \delta_0^{-1} =  L^{\frac{11}{12}}.
$
\end{proof}

We are in position to finish the proof of Lemma \ref{outside}.
If $f^nx$
is free (this includes the case $f^nx \in C_\delta$),
then we use Lemma \ref{initial}(c) for bound segments and $|f'| >
K^{-1} L^{\frac{1}{12}}$ for iterates in free segments, which are
 out of $C_{\delta_0}$.
This proves Lemma \ref{outside}(b). If $f^nx$ is bound, then
there is a drop of a factor $> \delta$ at the last free return
that can not be recovered. In this case we need the factor
$\delta$ in Lemma \ref{outside}(a). \qed

\end{document}